\documentclass[letterpaper, 10 pt, conference]{ieeeconf} 
\usepackage{balance}
\IEEEoverridecommandlockouts
  \IEEEpeerreviewmaketitle
\usepackage{subfig}
\usepackage{float}
\usepackage{transparent}
\usepackage{amsmath}
\usepackage{amssymb}
\usepackage{graphicx}
\usepackage{color}
\usepackage{cite}
\usepackage{balance}

\usepackage{mathtools}

\mathtoolsset{showonlyrefs}

\def\<{\leqslant}           
\def\>{\geqslant}           

\def\[{[\![}
\def\]{]\!]}

\usepackage{transparent}
\usepackage{amsmath}
\usepackage{amssymb}
\usepackage{graphicx}
\usepackage{color}
\usepackage{fullpage}

\newcommand{\gamovertwo} {{\frac{\gamma^2}{2}}}
\newcommand{\gammovertwo} \gamovertwo
\newcommand{\gammuovertwo} \gamovertwo
\newcommand{\gammuepstovertwo} \gamovertwo
\newcommand{\gammubarepstovertwo} \gamovertwo

\newcommand{\noncr} {\nonumber\\}

\newcommand{\beasnum}{\begin{eqnarray}}
\newcommand{\eeasnum}{\end{eqnarray}}
\newcommand{\beas}{\begin{eqnarray*}}
\newcommand{\eeas}{\end{eqnarray*}}

\newcommand{\be}{\begin{equation}}
\newcommand{\ee}{\end{equation}}
\newcommand{\ba}{\begin{array}}
\newcommand{\ea}{\end{array}}

\newtheorem{theorem}            {Theorem}[section]

\newtheorem{sideremark}         [theorem]{Remark}
\newtheorem{sideeg}           [theorem]{Example}
\newtheorem{sideconj}           [theorem]{Conjecture}

\def\argmin                     {\mathop{\rm argmin}}

\def \aeq {&=}
\def \adefeq {&:=}
\newcommand{\qed} {\hskip 0.2em\lower 0.7ex\hbox{\vbox{\hrule
\hbox{\vrule height 1.2ex\hskip 0.4em\vrule height 1.2ex}
\hrule}}}

\def \sensorBall{\mathcal{S}^M}

\usepackage{color}

\def \figname {fig:minplusrobustest:}
\def \eqnname {eq:minplusrobustest:}

\def \secname {sec:minplusrobustest:}

\def \traX {\mathcal{X}}
\def \wtformeas {H}
\def \dpoperator {\mathbf{S}}
\author{
\authorblockN{Srinivas Sridharan}
\authorblockA{Dept. of Mechanical and Aerospace Engineering \\
University of California San Diego \\
Email: {srsridharan@eng.ucsd.edu}}
\thanks{Research partially supported by 
NSF grant DMS-0808131 and AFOSR.}
}

\title{Deterministic  filtering and max-plus methods for robust state estimation in multi-sensor settings}

\begin{document}
\maketitle
\begin{abstract}
A robust (deterministic) filtering approach to the problem of optimal sensor selection is considered herein. 
For a given system with several sensors, at each time step the output of one of the sensors must be chosen in order to obtain the best 
state estimate. We reformulate this problem in an optimal control framework which can then be solved using  dynamic programming.
In order to tackle the numerical  computation of the solution in an efficient manner, we exploit the preservation of the min-plus structure
of the optimal cost function when acted upon by the dynamic programming operator. 
This technique yields  a grid free numerical approach to the problem. Simulations on an example problem serve to  highlight the efficacy of this  generalizable  approach to robust  multi-sensor  state estimation.

\end{abstract}
\section{Introduction}\label{\secname introduction}

Various problems of practical interest involve the task of estimating the state of the system 
of interest under under the presence of a  number of sensors. For example, a \lq sensor fusion\rq\, problem involves determining an optimal way to combine the measurements from multiple sensors each of which may measure a subset of the states of the system.   This  is, in essence, a state estimation problem. 

 A more involved problem arises when the state estimation strategy is given the freedom to 
dynamically select the set of sensors used. Such a \lq sensor scheduling\rq\, problem can be 
used to model scenarios encompassing robust estimation under potential sensor faults;  in addition this framework may  be used to preselect a set of sensors to be used, based on constraints  on the available communication bandwidth or energy limitations on transmitting sensed information.

In the domain of optimal state estimation, there are two distinct approaches to estimation/filter design which have been applied in the design of estimators for 
systems subjected to process and measurement noise.  The most 
well known approach to filter design is the Kalman filter developed
for linear systems in the 1960\rq s \cite{kalman1960new}. 
The Kalman filter
uses the statistics of the noise/measurement processes in order
to compute the weightings on the measurement error, used to update the
state estimate, as new measurements are observed. An alternative approach
is the minimum energy filtering technique developed by Mortensen \cite{mortensen1968}. This interpretation views the optimal state estimation problem as an
optimal control problem wherein the objective is to minimize the energy
of the noise process required to explain the observations.
The resulting filter obtained via both these approaches \textit{for the 
linear system case with white noise processes} (and a quadratic
energy function with $L^2$ noise processes) turns out to be the 
Kalman filter \cite{jazwinski1970stochastic,hijab1980minimum,willems2004deterministic}.

The articles \cite{rago1996censoring,miller1997optimization}  address the sensor scheduling problem in the case of linear systems with Gaussian input 
noise.  An  alternative approach for optimal sensor scheduling in nonlinear systems, formulated via imposing a cost on the switching of sensor choice and the error  in estimation due to this choice, was undertaken in \cite{baras1989optimal}. 
The solution for this case depends on the solution of a Riccati differential equation 
arising from the choice of sensors, and can be solved offline. Generalized versions of this problem
for the case of uncertain systems with linear (albeit time varying) dynamics using integral quadratic
constraints to model the uncertainty of the system were studied in \cite{savkin2001problem}. The latter uses a 
deterministic filtering interpretation of the estimator (c.f. \cite{fleming1997deterministic,hijab1980minimum}).

In this article we study an approach to solving problems of sensor scheduling for systems with 
linear dynamics which involves solving a modified version of the deterministic filtering problem. The organization of the 
article is as follows. 
In Sec.~\ref{\secname sysAndProbDescription}  we  formulate the estimation problem as an 
associated optimal control problem which may be solved via dynamic programming. Hence, the methods employed herein  generalize  deterministic filtering to deal with sensor selection strategies.
 In Sec.~\ref{\secname minplusandstructpres} we describe how the Hamilton-Jacobi-Bellman (HJB) equation may be solved using a new class of computational approaches to solving such partial differential equations.  
  The underlying idea involves 
 modeling the value function as an element of a semi-convex function space whose structure
is preserved under the dynamic programming operator.  Such a technique belongs to a class of recently developed methods termed max-plus methods which have led to intriguing advancements 
in computational tractability for optimal control problems in several domains \cite{fleming2000max,mceneaney2006mpm,fleming2010max,sjmmcdc2010,mceneaney2011idempotent}.
 The techniques introduced herein can be extended to certain classes of nonlinear dynamics and nonlinear output functions as well \cite{kallapur2012min}.  
In Sec.~\ref{\secname example} we demonstrate the application of the theory  developed, using an example system with five sensors. We then conclude with a description of problems of interest and avenues for further exploration in Sec.~\ref{\secname conclusion}.
%


%
%

\section{System description}\label{\secname sysAndProbDescription}
In this article we consider the class of systems with discrete time dynamics of the form
\begin{align}
x_{k+1} \aeq \tilde{A} x_k + \tilde{B} u_k + w_k, \qquad x_k , w_k \in \mathbb{R}^n, \label{\eqnname fwdsyseqn}
\end{align}
where $x_{k}$ denotes the system state at a time instant $k$ and $w_k$ is a disturbance signal.
The output equation for the system with $M$ sensors is 
\begin{align}
y^j_k \aeq C^j x_k + \eta^j_k,   \qquad j \in \{1, 2, \ldots M\}. 
\end{align}
where $\eta^j_k$ is  measurement noise in the $j$ th sensor at time step $k$. We assume that
the $M$ sensors above are such that the pairs $(A, C^j)$ are observable for each $j$. 

In the results to be presented in this article, we make use of a time reversed representation of the dynamics rather than the time incremental version in \eqref{\eqnname fwdsyseqn}. This  time-inverted system has the model
\begin{align}
x_{k+1} \aeq {A} x_k + {B} u_k + w_k, \qquad x_k , w_k \in \mathbb{R}^n, \label{\eqnname syseqn}
\end{align}
and we denote the solution to this dynamics at any time step $k$ given the state at a different time instant $j$, a control 
signal $u$ and a disturbance $w$ as $$\traX (j|x_{k+1}=x, u(\cdot),w).$$
For this $M$ sensor system we pose the dynamic sensor selection problem as a modification of the 
deterministic filtering problem as follows. We commence by modeling the choice of sensors at each 
time step  as an element $\lambda$ of the unit ball $\sensorBall$. Specifically, if only one sensor 
say the $q$ th sensor) is chosen, then the vector $\lambda$ is an $M$ dimensional vector with $1$ 
at the $q$-th element and $0$ at the others. Hence $\lambda$ has a probabilistic interpretation 
in terms of the relative belief in each sensor\rq s accuracy. 

For any particular disturbance signal realization and the control signal  $\lambda(\cdot)$  for the choice of sensor, we 
consider a cost function of the form
\begin{align}
J(\lambda, w, x) \aeq \Big\{ \sum_{k=0}^K{(w_k)^T (w_k)} + \|\hat{x}(0) - \tilde{x}_0\|^2_L +
\noncr & \sum_{k=0}^K{\sum_{j = 0}^M \|y^j_k - C^j x_k\|^2_\wtformeas  \lambda^j_k },  \Big\} \label{\eqnname costfn}
\end{align}
with weighting matrices $\wtformeas$ and  $L$ on the measurement error and state estimation error respectively.
Here $\tilde{x}_0$ is the assumed value of the initial state of the system and $\hat{x}(k)$ is the estimated state of the system at time $k$ under the dynamics   \eqref{\eqnname syseqn}. Hence, the cost function penalizes the disturbance, deviations from the system dynamics (in terms of the  
deviation in the initial states and the measurement error). This form of the cost function differs from the standard form of the robust (deterministic) filter \cite{fleming1997deterministic}  in the following manner. Here the choice of  the sensors is incorporated into the cost of using the final term in \eqref{\eqnname 
costfn}. Note that the choice $\lambda^k_j$ of sensor beliefs is allowed to vary at each time step. We may also penalize changes in the choice of sensors as was done in \cite{baras1989optimal} or study a
one time step finite horizon control problem as in \cite{savkin2001problem}. 

The optimal cost function for this problem is 
\begin{align}
V_{k}(x) \adefeq \inf_{\substack{\lambda \in {\sensorBall}(\cdot)\\ w \in \mathcal{W}}} J_k(x,\lambda,w),
\end{align}
where $\sensorBall(\cdot)$ and $\mathcal{W}$ are the piecewise continuous signals taking up values in $\sensorBall$ and $\mathbb{R}$ respectively. 
The value function at any time $k$ leads directly to the optimal  state estimate  via the relation
\begin{align}
x^*_k := \arg \inf_{x \in \mathbb{R}^n} V_k(x).
\end{align}

Now, we apply the dynamic programming principle from optimal control theory to obtain 
\begin{align}
V_{k+1}(x) \aeq \inf_{\substack{\lambda \in \sensorBall\\ w \in \mathbb{R}^n}} \Big\{ V_{k}(\traX (k|x_{k+1}=x, u, w) \noncr &+ \|w\|^2 + \sum_{j = 0}^M \|y^j_{k+1} - C^j x\|^2_\wtformeas  \lambda^j_k  \Big\}.
\end{align}
This equation, can be cast in the form of an operator $\dpoperator$ acting on the value function as follows:
\begin{align}
V_{k+1}(x) \aeq \dpoperator\Big[V_k\Big](x). \label{\eqnname dpoperatordef}
\end{align}

Standard approaches to solve for the optimal cost function (value function) involve the solution of the associated Hamilton-Jacobi-Bellman partial differential equation which is the infinitesimal limit
of the above equation. However, this leads to issues of large computational complexity, especially in cases with nonlinear dynamics and a higher number of dimensions of the state space. In this work, we apply a technique which has shown much promise in enabling the solution to certain classes of optimal control/ optimal filtering problems \cite{fleming2010max,mceneaney2008cdf} -- for both linear and nonlinear systems. This approach makes use
of the fact that the dynamic programming operator, \eqref{\eqnname dpoperatordef} above, is linear on the space of semi-convex functions. Thus if we prove that the structure of the value function is preserved under the dynamic programming operator and the parameterization can be done independent  of the state space (i.e. the parameters of the value function do not depend on the point in space at which the 
value function is evaluated), then we can obtain a numerically efficient technique to solve optimal 
control problems which can be cast into this form by repeated applications of the dynamic programming operator (while storing only the parameter values).  Further background details and 
specific applications of this concept to various problems can be found in \cite{fleming2000max,kolokol2001idempotent,akian2005solutions,mceneaney2008cdf,mceneaney2009complexity,sridharanMtns2010,mceneaney2011idempotent}.

\section{The min-plus expansion and the propagation of the cost function}\label{\secname minplusandstructpres}
In this section we demonstrate that under one time step of propagation by the dynamic programming operator (as in \eqref{\eqnname dpoperatordef}) a specific form of the value function  is preserved. As indicated previously, this feature leads to an efficient approach to solve for the value function (and hence, to the state estimate). The value function thus determined may be used to obtain the  optimal state estimate.   We start with a description of the solution technique for the cost function

\subsection{Min-plus approach to solving the  optimal control problem}
Motivated by the terminal form of the value function 
\begin{align}
V_0(x) = {\|x - \tilde{x}_0\|_L}^2,
\end{align}
we assume that the value function at any time step $k$ can be written as follows
\begin{align}
V_k(x) = \bigoplus_{\alpha \in \Gamma_k} \Big\{ x^T Q^{\alpha}_a x + 2 Q^{\alpha}_b x + Q^{\alpha}_c \Big\}.\label{\eqnname quadformofvalfn}
\end{align}
Here $\Gamma_k$ is the index set at the time instant $k$ and the notation $\bigoplus$ denotes the minimization operation over the index set. We assume that the quadratics associated with the index set are convex, leading to a well defined state estimate.
\begin{theorem}
Given the form \eqref{\eqnname quadformofvalfn} of the value function at time $k$, the propagation under the dynamic programming operator  \eqref{\eqnname dpoperatordef} leads to a preservation in structure of the value function viz. there exists a set $\Gamma_{k+1}$ and spatially invariant  parameters 
$R^\beta_a$, $R^\beta_b$,  $R^\beta_c$,  for $\beta \in \Gamma_{k+1}$ such that the value function  at time $k+1$  is
\begin{align}
V_{k+1}(x) = \bigoplus_{\beta \in \Gamma_{k+1}} \Big\{ x^T Q^{\beta}_a x + 2 Q^{\beta}_b x + Q^{\beta}_c \Big\}.\label{\eqnname quadformofvalfnafteronetimestep}
\end{align}
\end{theorem}
\begin{proof}
Assuming the form of the value function in \eqref{\eqnname quadformofvalfn}, we note that the 
action of the propagation equation \eqref{\eqnname dpoperatordef} leads to following recursion 
\begin{align}
V_{k+1}(x|u) \aeq  \noncr \inf_{\substack{w \in W \\ \lambda \in \sensorBall }} \Big\{ V_k(\traX(k|x_{k+1}&=x,u_
{k+1} = u,w_{k+1}=w))\ldots   \noncr + &  \|w\|^2 +    \sum_{j = 0}^M \|y^j_{k+1} - C^j x\|^2_\wtformeas  \lambda^j 
\Big\},
\noncr
\aeq \inf_{\substack{w \in W \\ \lambda \in \sensorBall }}\Big\{   V_k({A} x + {B} u + w)+ \ldots \noncr &  \|w\|^2 +    \sum_{j = 0}^M \|y^j_{k+1} - C^j x\|^2_\wtformeas  \lambda^j \Big\}. \label{\eqnname eqvalfnrec2}
\end{align}
Using the form of the value function \eqref{\eqnname quadformofvalfn} in  \eqref{\eqnname eqvalfnrec2} we have
\begin{align}
V_{k+1}(x|u) = \qquad \qquad & \noncr \inf_{\substack{w \in W \\ \lambda \in \sensorBall }}\Big\{\bigoplus_{\alpha \in \Gamma_k} \Big\{ ({A} x + {B} u &+ w)^T Q^{\alpha}_a ({A} x + {B} u + w) + \noncr
 2 Q^{\alpha}_b ({A} x + {B} u + w) + Q^{\alpha}_c \Big\}&
+   \|w\|^2 +  \ldots \noncr    \sum_{j = 0}^M \|y^j_{k+1} &- C^j x\|^2_\wtformeas  \lambda^j \Big\} \label{\eqnname valiter}
\end{align}
We note that for each $j$, since $w$ and $\|y^j- C^j x\|$ are independent, therefore the optimal 
value of $\lambda$  can be determined from the last term in \eqref{\eqnname valiter}.
Specifically the optimal value of a linear function of $\lambda$ on a compact set is given by a  point on the boundary of the compact set. This leads to one of the sensors being used in one time step with the others being turned off (in the case where the sensor dynamics can be altered at each time step). The optimal argument $\lambda^*$ has one element $\lambda_j$ equal to $1$  with all other elements of $\lambda$ set to $0$, where 
\begin{align}
 j:= \mathop{argmin}_j \|y^j_{k+1} - C^j x\|^2_Q \label{\eqnname proofofformpres0}
\end{align}
For simplicity of notation we denote 
\begin{align}
z : = Ax+Bu. \label{\eqnname redeffz}
\end{align}   
Note that for ever value of the state $x$, the value of $w$ which minimizes the quadratic portion of the cost function dependent on $w$ can be obtained analytically as follows. For a triple $(Q^\alpha_a,Q^\alpha_b, Q^\alpha_c)$, the corresponding  $w^*$ which minimizes the term 
\begin{align}
 (z+w)^T Q^{\alpha}_a (z+w) + 2 Q^{\alpha}_b (z+w) +  \noncr \qquad  Q^{\alpha}_c  +  \|w\|^2 \label{\eqnname proofofformpres1}
\end{align}
in the value function, can be shown to have the analytic form
\begin{align}
w^* = (I + Q^\alpha_a)^{-1} \Big[ -Q^\alpha_b - Q^\alpha_a z\Big].\label{\eqnname proofofformpres2}
\end{align}
Hence, for every $\alpha \in \Gamma_k$, there exists a disturbance value $w^*$ which can be optimally chosen to minimize the quadratic form ( in $w$) in the cost function. Substituting this value of $w^*$ in \eqref{\eqnname proofofformpres1} and using \eqref{\eqnname redeffz},  we obtain a quadratic form in $x$ which can be written as
\begin{align}
\bigoplus_{\alpha \in \Gamma_k} x^T P^\alpha_a x + 2 P^{\alpha}_b x + P^{\alpha}_c, \label{\eqnname valitereqn4}
\end{align}
where from    \eqref{\eqnname redeffz}, \eqref{\eqnname proofofformpres1} and \eqref{\eqnname proofofformpres2} it can be seen that $P^\alpha_a$, $P^{\alpha}_b$, $P^{\alpha}_c$ are obtained from the coefficients of quadratic terms in $x$ in \eqref{\eqnname proofofformpres1}  and using the term $Q^\alpha_a$, $Q^\alpha_b$, $Q^\alpha_c$ respectively.
Thus using \eqref{\eqnname proofofformpres0} and \eqref{\eqnname valitereqn4} the value iteration in \eqref{\eqnname valiter} simplifies to 
\begin{align}
V_{k+1}(x|u) \aeq  \bigoplus_{\alpha \in \Gamma_k} x^T P^\alpha_a x + 2 P^{\alpha}_b x + P^{\alpha}_c  +  \noncr  & \bigoplus_j \|y^j_{k+1} - C^j x\|^2_\wtformeas \label{\eqnname proofofformpres5}
\end{align}
Creating a new index set $\Gamma_{k+1} =  \Gamma_k \times \mathcal{M}$ and combining
 coefficients of equal powers of $x$, we can rewrite  \eqref{\eqnname proofofformpres5} as 
\begin{align}
V_{k+1}(x|u) \aeq  \bigoplus_{\beta \in \Gamma_{k+1}} x^T R^\beta_a x + 2 R^{\beta}_b x + R^{\beta}_c ,
\end{align}
where it can be seen that $R^\beta_a$ is obtained from $P^\alpha_a$ and terms from a subset of the $C^j$.  This completes the proof of the desired result.
\end{proof}

\subsection{Generating the state estimate}
We now indicate how the value function  obtained above is used to compute the optimal state estimate. 
Starting with the terminal cost, we apply the propagation procedure for $N$ time steps (i.e., for the entire duration of the estimation process) to obtain a value function of the form 
\begin{align}
V_{N}(x|u(\cdot)) \aeq  \bigoplus_{\gamma \in \Gamma_{N}} x^T R^\gamma_a x + 2 R^{\gamma}_b x + R^{\gamma}_c , \label{\eqnname vstateest}
\end{align}
The optimal state estimate at time step $N$  is defined as 
\begin{align}
x^*_N \aeq \argmin_{x \in \mathbb{R}^n} V_N(x|u(\cdot)). \label{\eqnname optstateest1}
\end{align}
 Now given \eqref{\eqnname vstateest}, and an unconstrained state space, we can solve for the state estimate \eqref{\eqnname  optstateest1} by determining the minimum for each quadratic in the basis expansion  \eqref{\eqnname vstateest}.  Thus
 \begin{align}
x^*_N \aeq \argmin_{x \in \mathcal{X}_N} V_N(x|u(\cdot)), \label{\eqnname optstateest2}
\end{align}
where 
\begin{align}
 \mathcal{X}_N \adefeq \{ - [ R^\gamma_a]^{-1}  R^\gamma_b, \quad \forall \gamma \in  \Gamma_{N}\}.
\end{align}
The set $ \mathcal{X}_N$ can be seen to be the set of the minima of the set of convex quadratic functions  in the expansion 
\eqref{\eqnname vstateest}.
\section{Example} \label{\secname example}
In this section we validate the theory developed thus far in this article, using an example problem as follows. 
We consider a first order linear system with forward time dynamics
\begin{align}
x_{k+1} \aeq A_{f} x_k + B_f u_k + B_{wf} w_k, \\
A_f &= 0.7, \quad B_f = 0.4, \quad B_{wf} = -0.7.
\end{align}
Without loss of generality, the value of $B_{wf}$ has been chosen in order to ensure the form of the backward dynamics in \eqref{\eqnname syseqn} i.e., the coefficient of the disturbance term becomes $1$.  We assume the presence of five sensors with linear output maps 
\begin{align}
y^j_k = C^j_k x_k,
\end{align}
where the $C^j$ belong to the set $\{1.5, -2, 1.7, 3.5, 1 \}$ (under nominal conditions). In 
order to apply the formulation of the filter in  Sec.~\ref{\secname minplusandstructpres} to this example, we assume weighting 
matrices $L = 5$, $H =3$. We apply the max-plus deterministic filter to this example system using a filter horizon of $5$ time steps. Assuming the case of a systemic failure of all but one of the sensors  - i.e. the set of $C^j$ becomes $\{0.01, 0.1, 0.2, 0.01, 1 \}$, we obtain the results indicated in
Fig.~\ref{\figname fig1} for the case of no errors in the initial state assumption and Fig.~\ref{\figname fig2} for the case of an initial state error (in the estimate).


\begin{figure}[htbp]
   \centering
   \subfloat[State estimation]{\includegraphics[width=0.55\textwidth]{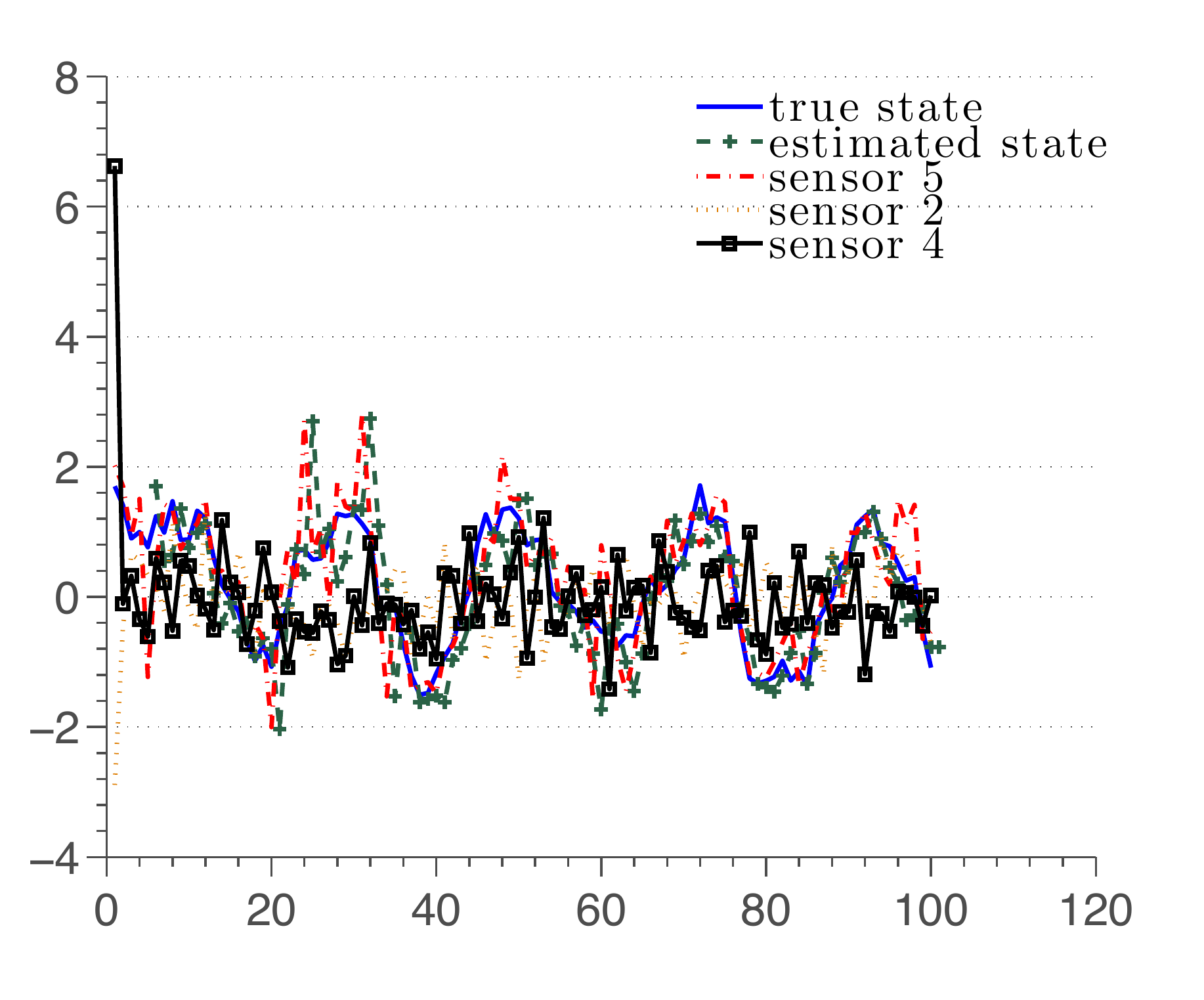}} \\         
  \subfloat[Sensor selected]{\includegraphics[width=0.5\textwidth]{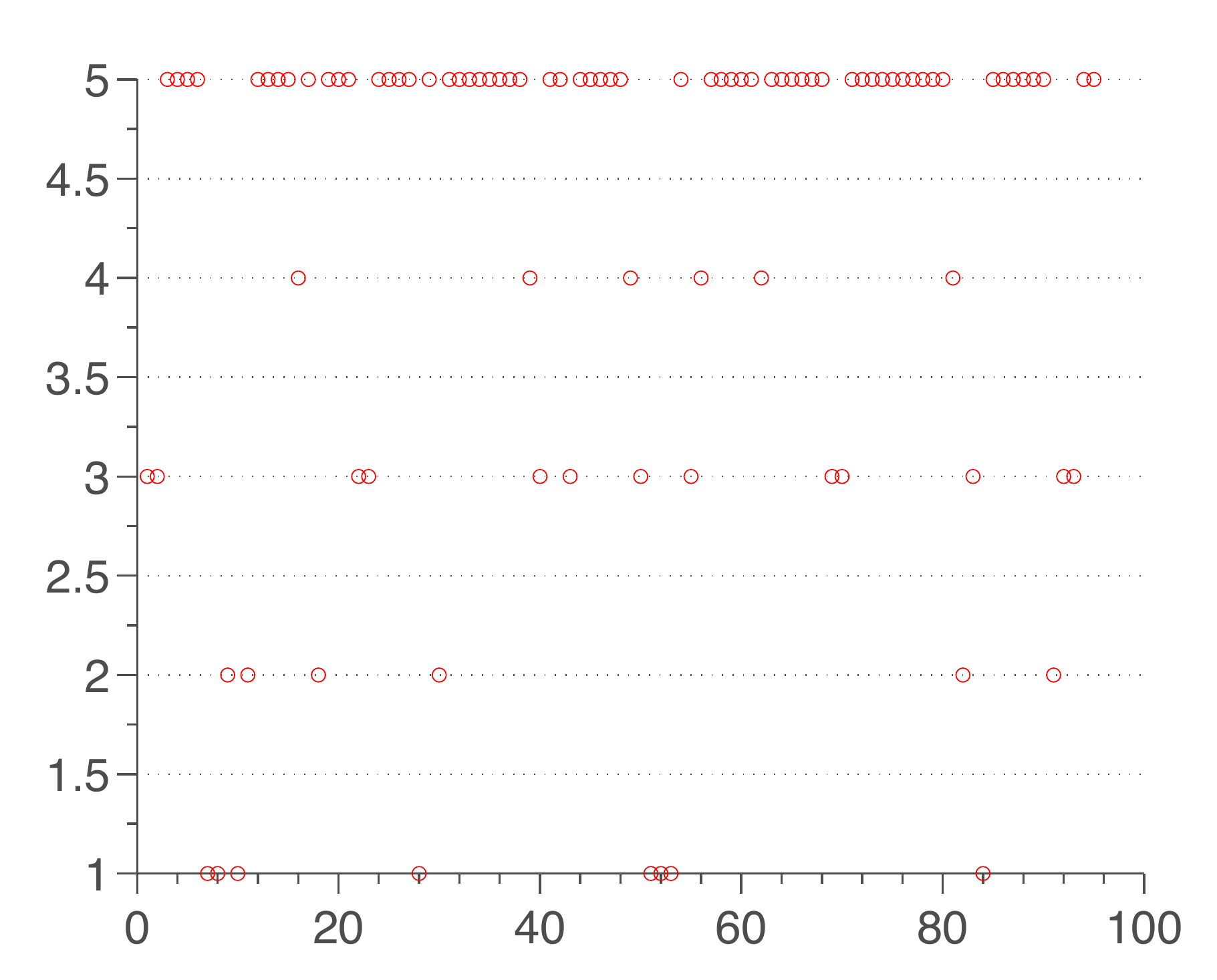}}
  \caption{The case of zero initial estimation error}
  \label{\figname fig1}
\end{figure}

\begin{figure}[htbp]
   \centering
   \subfloat[State estimation]{\includegraphics[width=0.55\textwidth]{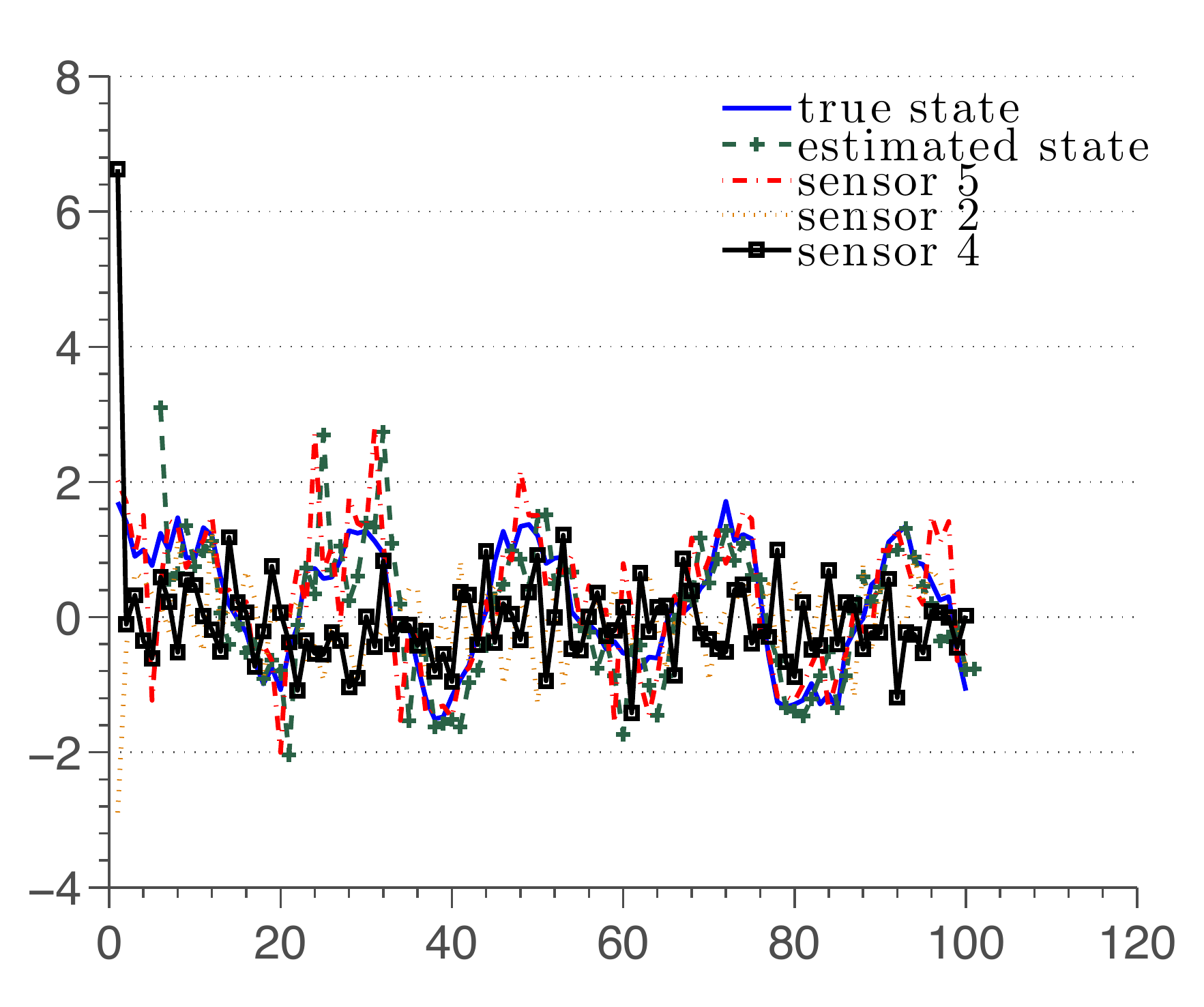}} \\         
  \subfloat[Sensor selected]{\includegraphics[width=0.5\textwidth]{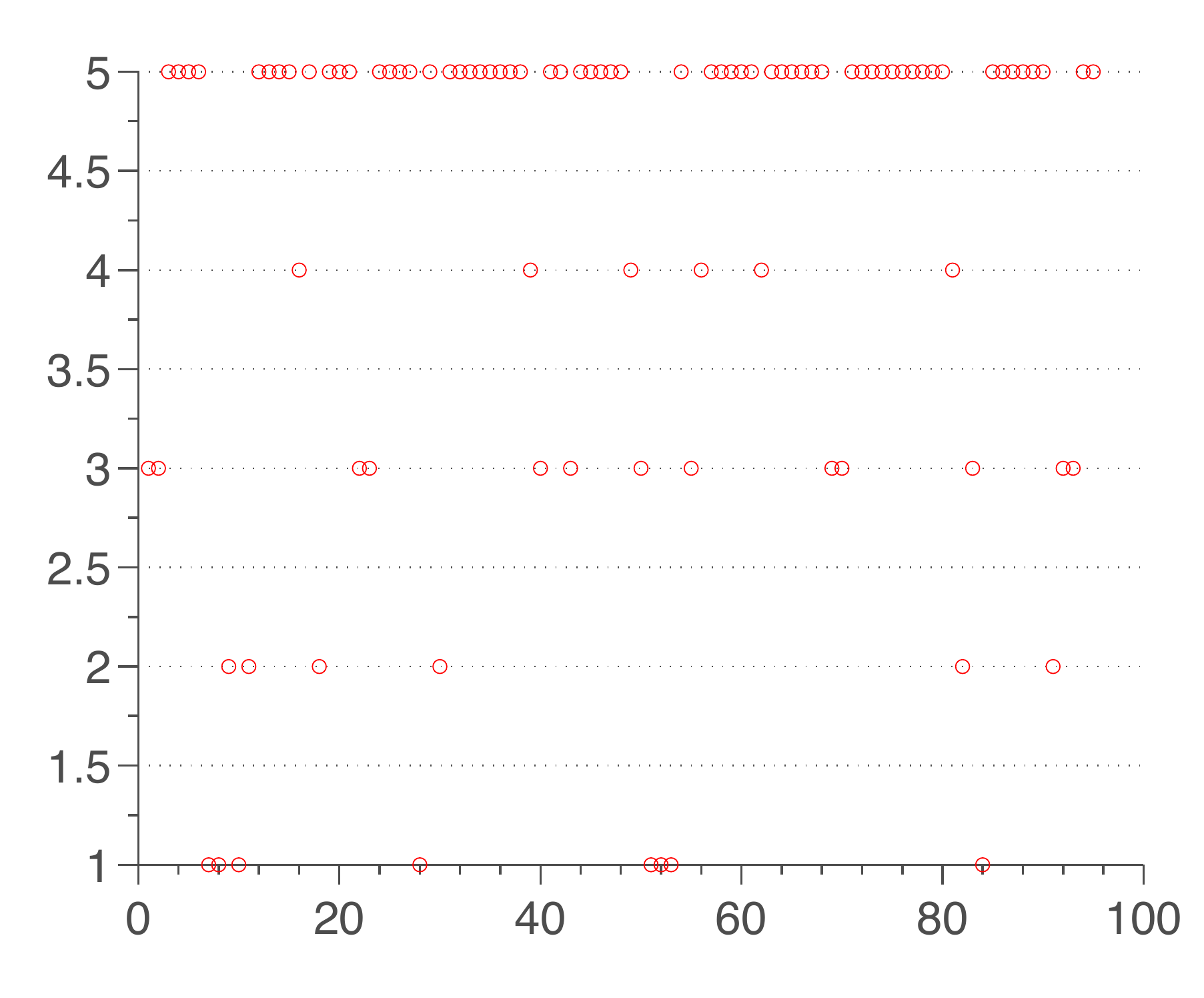}}
  \caption{The case of an initial estimation error}
  \label{\figname fig2}
\end{figure}

%
%

Thus it is seen that the estimation mechanism solves the  
 sensor selection problem in a robust manner. The algorithm uses the sensor which provides measurement signals which are used to minimize the cost function (which thereby yields a state estimate).  
It can be seen that the approach recovers the state despite noisy measurements and sensor failures.  A salient feature of the behavior of the estimation approach is that in both scenarios 
studied in the example (zero/non-zero initial estimation error), when the noise in the functioning sensor increases rapidly over a short time duration, the algorithm temporarily switches to a different
sensor which yields a lower cost function.  Once the functioning sensor yields a better performance, it is once again used for estimation as expected. 
\section{Conclusions and Future Directions} \label{\secname conclusion}
In this article we introduced the application of the deterministic filtering concept to sensor 
selection for robust state estimation. The resulting optimal control problem yielded a HJB equation 
which was solved via the max-plus approach which utilizes the linearity of the dynamic programming propagator in order to obtain an efficient method to solve the HJB equation at each time step. 
The technique yielded promising results when applied to the case of drastic sensor failures for a 
linear system.

 Some  fruitful  directions along which further research into this topic may be directed are: the extension of these methods to solve the sensor scheduling problem for nonlinear 
systems, to understand the behavior of the approach during short duration failures in sensors, 
the case of unmodeled system dynamics, incorporation of prior belief regarding sensor failure rates into the algorithm, the fusion of information from multiple sensors when not all of them satisfy the observability condition, the case where switching costs are imposed on  changes in the sensor vector  etc.  One other facet of this approach to sensor scheduling is that it benefits
from a non-zero input signal $u(\cdot)$. Thus it is useful to understand the requirement for persist excitation like conditions on the input to the system,  in order to ensure robust state estimation.

\balance
\bibliographystyle{unsrt}


\begin{thebibliography}{10}

\bibitem{baras1989optimal}
J.S. Baras and A.~Bensoussan.
\newblock Optimal sensor scheduling in nonlinear filtering of diffusion
  processes.
\newblock {\em SIAM journal on control and optimization}, 27(4):786--813, 1989.

\bibitem{rago1996censoring}
C.~Rago, P.~Willett, and Y.~Bar-Shalom.
\newblock Censoring sensors: A low-communication-rate scheme for distributed
  detection.
\newblock {\em Aerospace and Electronic Systems, IEEE Transactions on},
  32(2):554--568, 1996.

\bibitem{miller1997optimization}
B.M. Miller and W.J. Runggaldier.
\newblock Optimization of observations: a stochastic control approach.
\newblock {\em SIAM journal on control and optimization}, 35(3):1030--1052,
  1997.

\bibitem{savkin2001problem}
A.V. Savkin, R.J. Evans, and E.~Skafidas.
\newblock The problem of optimal robust sensor scheduling.
\newblock {\em Systems \& Control Letters}, 43(2):149--157, 2001.

\bibitem{fleming1997deterministic}
W.H. Fleming, E.~De~Giorgi, Lefschetz~Center for Dynamical~Systems, Brown
  University.~Center for Control~Sciences, and Brown University.~Division
  of~Applied~Mathematics.
\newblock Deterministic nonlinear filtering.
\newblock {\em Annali della Scuola Normale Superiore di Pisa-Classe di
  Scienze-Serie IV}, 25(3):435--454, 1997.

\bibitem{hijab1980minimum}
O.B. Hijab.
\newblock {\em Minimum energy estimation}.
\newblock PhD thesis, University of California, Berkeley, 1980.

\bibitem{fleming2000max}
W.H. Fleming and W.M. McEneaney.
\newblock {A Max-Plus-Based Algorithm for a Hamilton--Jacobi--Bellman Equation
  of Nonlinear Filtering}.
\newblock {\em SIAM Journal on Control and Optimization}, 38:683, 2000.

\bibitem{mceneaney2006mpm}
W.M. McEneaney.
\newblock {\em {Max-plus methods for nonlinear control and estimation}}.
\newblock Birkhauser, 2006.

\bibitem{fleming2010max}
W.H. Fleming, H.~Kaise, and S.J. Sheu.
\newblock Max-plus stochastic control and risk-sensitivity.
\newblock {\em Applied Mathematics and Optimization}, 62(1):81--144, 2010.

\bibitem{sjmmcdc2010}
Srinivas Sridharan, Mile Gu, Matthew~R. James, and William M.~Mc Eneaney.
\newblock An efficient computational method for the optimal control of higher
  dimensional quantum systems.
\newblock In {\em Proceedings of the, 49th IEEE Conference on Decision and
  Control}, pages 2996--3001, Atlanta, USA, December 15-17 2010.

\bibitem{mceneaney2011idempotent}
W.M. McEneaney.
\newblock Idempotent method for deception games and complexity attenuation⋆.
\newblock In {\em World Congress}, volume~18, pages 4453--4458, 2011.

\bibitem{kallapur2012min}
A.G. Kallapur, S.~Sridharan, W.M. McEneaney, and I.R. Petersen.
\newblock Min-plus techniques for set-valued state estimation.
\newblock { (to be presented at the  Control and Decision Conference, Dec. 2012)}, 2012.

\bibitem{kalman1960new}
R.E. Kalman et~al.
\newblock A new approach to linear filtering and prediction problems.
\newblock {\em Journal of basic Engineering}, 82(1):35--45, 1960.

\bibitem{mortensen1968}
R.~E. Mortensen.
\newblock Maximum-likelihood recursive nonlinear filtering.
\newblock {\em Journal of Optimization Theory and Applications}, 2:386--394,
  1968.
\newblock 10.1007/BF00925744.

\bibitem{jazwinski1970stochastic}
A.H. Jazwinski.
\newblock {\em Stochastic processes and filtering theory}, volume~63.
\newblock Academic Pr, 1970.

\bibitem{willems2004deterministic}
JC~Willems.
\newblock Deterministic least squares filtering.
\newblock {\em Journal of econometrics}, 118(1):341--373, 2004.

\bibitem{kolokol2001idempotent}
V.N. Kolokoltsov.
\newblock Idempotent structures in optimization.
\newblock {\em Journal of Mathematical Sciences}, 104(1):847--880, 2001.

\bibitem{akian2005solutions}
M.~Akian, S.~Gaubert, and V.~Kolokoltsov.
\newblock Solutions of max-plus linear equations and large deviations.
\newblock In {\em Decision and Control, 2005 and 2005 European Control
  Conference. CDC-ECC'05. 44th IEEE Conference on}, pages 7787--7792. IEEE,
  2005.

\bibitem{mceneaney2008cdf}
W.M. McEneaney.
\newblock {A curse-of-dimensionality-free numerical method for solution of
  certain HJB PDEs}.
\newblock {\em SIAM Journal on Control and Optimization}, 46(4):1239--1276,
  2008.

\bibitem{mceneaney2009complexity}
W.M. McEneaney.
\newblock Complexity reduction, cornices and pruning.
\newblock In {\em Proc. of the International Conference on Tropical and
  Idempotent Mathematics, GL Litvinov and SN Sergeev (Eds.), Contemporary
  Math}, volume 495, pages 293--303, 2009.

\bibitem{sridharanMtns2010}
Srinivas Sridharan, W.M. McEneaney, and Matthew.R. James.
\newblock Curse-of-dimensionality-free control methods for the optimal
  synthesis of quantum circuits.
\newblock In {\em Mathematical Theory of Networks and Systems (MTNS)}, 2010.

\end{thebibliography}


\end{document}